\documentclass[a4paper]{scrartcl}
\usepackage[latin1]{inputenc}
\usepackage{amsmath}
\usepackage{bbm}
\usepackage{amssymb}
\usepackage{amsthm}
\usepackage{pdfsync}
\usepackage{esint}
\newtheorem{lemma}{Lemma}[section]
\newtheorem{theorem}{Theorem}[section]

\theoremstyle{remark}
\newtheorem{remark}{Remark}[section]
\newcommand{\cotan}{\operatorname{cotan}}

\numberwithin{equation}{section}

\begin{document}
\title{A Note on Extended Binomial Coefficients}
\author{Thorsten Neuschel
  \thanks{Department of Mathematics, KU Leuven, Celestijnenlaan 200B box 2400, BE-3001 Leuven, Belgium. E-mail: Thorsten.Neuschel@wis.kuleuven.be}}

\date{}

\maketitle

\paragraph{Abstract} We study the distribution of the extended binomial coefficients by deriving a complete asymptotic expansion with uniform error terms. We obtain the expansion from a local central limit theorem and we state all coefficients explicitly as sums of Hermite polynomials and Bernoulli numbers.


\section{Introduction}
The extended binomial coefficients, occasionally called polynomial coefficients (e.g., \cite[p.\ 77]{Comtet}), are defined as the coefficients in the expansion
\begin{equation}\label{ext}\sum_{k=0}^{\infty}\binom{n}{k}^{(q)} x^k = \left(1+x+x^2+\cdots+x^{q}\right)^n,\quad n,q \in \mathbb{N}=\{1,2,\ldots\}.
\end{equation}
In written form, they presumably appeared for the first time in works by De Moivre \cite[p.\ 41]{DeMoivre} and later they also were addressed by Euler \cite{Euler}. Since then, the extended binomial coefficients played a role mainly in the theory of compositions of integers as the number \(c(k,n,q)\) of compositions of \(k\) with \(n\) parts not exceeding \(q\) is given by
\[c(k,n,q)=\binom{n}{k-n}^{(q-1)}. \]
Thus, the extended binomial coefficients and their modifications have been studied in various papers and from different perspectives \cite{Andrews, Balakrishnan, Banderier, Caiado, Eger, Eger2, Fahssi, Heubach, Knopfmacher, Star}, and among the properties their distribution is of particular interest. Recently, Eger \cite{Eger2} showed (using a slightly different notation) that
\[\binom{n}{nq/2}^{(q)}\sim \frac{(q+1)^n}{\sqrt{2\pi n \frac{q(q+2)}{12}}},\]
as \(n\rightarrow\infty\), meaning that the quotient of both sides tends to unity. Moreover, based upon numerical simulations \cite{Eger2} the question arises how well those coefficients can be approximated by ``normal approximations'' in general. It is the aim of this note to give a precise and comprehensive answer to this question by establishing a complete asymptotic expansion for the extended binomial coefficients with error terms holding uniformly with respect to all integer \(k\). More precisely, we show the following.
\begin{theorem} For all integers \(N\geq 2\) we have
\[\sqrt{\frac{q(q+2) n}{12}}\frac{1}{(1+q)^n} \binom{n}{k}^{(q)}=\frac{1}{\sqrt{2 \pi}} e^{-x^2/2}+\sum_{\nu=1}^{\left[(N-2)/2\right]} \frac{q_{2\nu}(x)}{n^{\nu}} +o\left(\frac{1}{n^{(N-2)/2}}\right),\]
as \(n\rightarrow \infty\), uniformly with respect to all \(k\in\mathbb{Z}\), with
\[x=\frac{\sqrt{12}}{\sqrt{q(q+2) n}}\left(k-\frac{q}{2}n\right),\]
\end{theorem}
where the functions \(q_{2\nu}(x)\) are given explicitly as sums of Hermite polynomials and Bernoulli numbers (see Theorem \ref{Main} below for the exact formulae).
Although we only deal with the very basic situation of the extended binomial coefficients in (\ref{ext}) here, the presented approach is a general one, which admits the derivation of (complete) asymptotic expansions in many applications. However, it is not always possible to obtain the involved quantities in a very explicit form, which is an instance making the case of the extended binomial coefficients further interesting and worth to be presented.

\section{Proof of the main result}
First of all we fix some notations following Petrov \cite{Petrov}. For a (real) random variable \(X\) we denote its characteristic function by
\[\varphi_X (t)=Ee^{itX}, \quad t\in \mathbb{R},\]
where, as usual, \(E\) means the mathematical expectation with respect to the underlying probability distribution. If \(X\) has finite moments up to \(k\)-th order, then \(\varphi_X\) is \(k\) times continuously differentiable on \(\mathbb{R}\) and we have
\[\frac{d^k}{dt^k} \varphi_X (t) \Big\vert_{t=0} = \frac{1}{i^k} EX^k.\]
Moreover, in this case we define the cumulants of order \(k\) by
\[\gamma_k =\frac{1}{i^k} \frac{d^k}{dt^k} \log \varphi_X (t)\Big\vert_{t=0},\]
where the logarithm takes its principal branch. Now, let \(\left(X_n\right)\) be a sequence of independent integer-valued random variables having a common distribution and suppose that for all positive integer values of \(k\) we have
\[E\vert X_1\vert^k < \infty\]
and
\[EX_1=\mu,\quad Var X_1 =\sigma^2 >0. \]
Thus, for the sum given by
\[S_n=\sum_{\nu=1}^n X_{\nu}\]
we obtain
\[ES_n=n \mu,\quad Var S_n = n\sigma^2,\]
and for integer \(k\) we define the probabilities
\[p_n(k)=P\left(S_n =k\right).\]
Furthermore, we introduce the Hermite polynomials (in the probabilist's version)
\[H_m (x)=(-1)^m e^{x^2/2} \frac{d^m}{dx^m} e^{-x^2/2},\]
and for positive integers \(\nu\) we define the functions
\begin{equation}\label{qq}q_{\nu} (x) = \frac{1}{\sqrt{2 \pi}} e^{-x^2 \slash 2} 
\sum_{k_1, \ldots, k_{\nu} \geq 0 \atop k_1 + 2 k_2 + \cdots + \nu k_{\nu} = \nu} ~
H_{\nu + 2s} (x) ~ \prod\limits^{\nu}_{m=1} ~ \frac{1}{k_m!} 
\left( \frac{\gamma_{m+2}}{(m+2)! \sigma^{m+2}} \right)^{k_m},
\end{equation}
where \(s = k_1 + \cdots + k_\nu\) and \(\gamma_{m+2}\) denotes the cumulant of order \(m+2\) of \(X_1\).

Finally, we demand (for convenience) that the maximal span of the distribution of \(X_1\) is equal to one. This means that there are no numbers \(a\) and \(h>1\) such that the values taken on by \(X_1\) with probability one can be expressed in the form \(a+hk\) (\(k\in\mathbb{Z}\)). Under all these assumptions we have the following complete asymptotic expansion in the sense of a local central limit theorem  \cite[p.\ 205]{Petrov}.
\begin{theorem}\label{EE} For all integers \(N\geq 2\) we have
\begin{equation}\label{EE1}\sigma \sqrt{n}p_n(k)=\frac{1}{\sqrt{2 \pi}} e^{-x^2/2}+\sum_{\nu=1}^{N-2} \frac{q_{\nu}(x)}{n^{\nu/2}} +o\left(\frac{1}{n^{(N-2)/2}}\right),
\end{equation}
as \(n\rightarrow \infty\), uniformly with respect to all \(k\in\mathbb{Z}\), where we have 
\[x=\frac{k-n\mu}{\sigma \sqrt{n}}.\]
\end{theorem}

In the following we choose \(X_1\) to take the integer values \(\{0,\ldots,q\}\) with 
\[P(X_1=k)=\frac{1}{q+1}, \quad k\in\{0,\ldots,q\}.\]
Hence, we obtain
\begin{equation}\label{P}p_n(k)=P(S_n=k)=\frac{1}{(1+q)^n} \binom{n}{k}^{(q)}, \quad k\in\mathbb{Z}.
\end{equation}
It is our aim to apply Theorem \ref{EE} in full generality and we want compute all cumulants as explicit as possible.

\begin{lemma} \label{CU}For the \(k\)-th order cumulant \(\gamma_k\) of \(X_1\) we have
\begin{equation}\label{GAMMA}\gamma_k= \begin{cases} \frac{q}{2}, &\text{if}~ k=1; \\
0, & \text{if}~ k ~\text{odd}~\text{and}~ k>1;\\
\frac{\mathcal{B}_{2l}}{2l} \left((q+1)^{2l}-1\right), & \text{if}~ k=2l, l\geq 1,\end{cases}
\end{equation}
where $\mathcal{B}_{\nu},\, \nu \geq 0$, denote the Bernoulli numbers, e.g., \cite[p.\ 22]{Grad}.
\end{lemma}
\begin{proof} First, we observe that the characteristic function of \(X_1\) is given by
\[\varphi_{X_1}(t)=\frac{1+e^{it}+\cdots+e^{qit}}{1+q}.\]
According to the definition of the cumulants we obtain for a positive integer \(k\)
\begin{align*}\gamma_k &=\frac{1}{i^k} \frac{d^k}{dt^k} \log \varphi_{X_1} (t)\Big\vert_{t=0}\\
&=\frac{1}{i^k} \frac{d^k}{dt^k} \left\{\log\left(1+e^{it}+\cdots+e^{qit}\right) -\log(1+q)\right\}\Big\vert_{t=0}\\
&=\frac{1}{i^k} \frac{d^k}{dt^k}\log\left(\frac{e^{(q+1)it}-1}{e^{it}-1}\right)\Big\vert_{t=0}\\
&=\frac{1}{i^k} \frac{d^k}{dt^k}\left\{\frac{q}{2}it+\log \left(\frac{\sin \frac{q+1}{2}t}{\sin \frac{t}{2}}\right)\right\} \Big\vert_{t=0}\\
&=\frac{q}{2}\delta_{k,1}+\frac{1}{i^k} \frac{d^k}{dt^k}\left\{\log \left(\frac{\sin \frac{q+1}{2}t}{ \frac{q+1}{2} t}\right)-\log \left(\frac{\sin \frac{t}{2}}{\frac{t}{2}}\right)\right\}\Big\vert_{t=0},
\end{align*}
where \(\delta_{k,1}\) denotes the Kronecker delta.
Using
\[\frac{d}{dz} \log \left(\frac{\sin z}{z}\right)=\cotan z -\frac{1}{z}\]
yields
\[\gamma_k=\frac{q}{2}\delta_{k,1}+\frac{1}{i^k} \frac{d^{k-1}}{dt^{k-1}}\left\{\frac{q+1}{2}\left(\cotan \frac{q+1}{2} t -\frac{2}{(q+1)t}\right)-\frac{1}{2}\left(\cotan \frac{t}{2}  -\frac{2}{t}\right)\right\}\Big\vert_{t=0}.\]
Now, making use of the following expansion (see, e.g., \cite[p.\ 35]{Grad})
\[\cotan z - \frac{1}{z} = \sum\limits^{\infty}_{m = 1} (-1)^{m} \frac{4^{m}}{(2 m)!} \mathcal{B}_{2 m} z^{2 m - 1}
~~,~~ 0 < |z| < \pi ,\]
after some algebra we obtain
\[\gamma_k=\frac{q}{2}\delta_{k,1}+\frac{1}{i^k} \frac{d^{k-1}}{dt^{k-1}}\sum\limits^{\infty}_{m = 1} (-1)^{m} \frac{\mathcal{B}_{2 m}}{(2 m)!}  \left((q+1)^{2m}-1\right)t^{2 m - 1}\Big\vert_{t=0}.\]
Carrying out the differentiation under the summation sign immediately gives us (\ref{GAMMA}).
\end{proof}

\begin{remark} As an immediate consequence of Lemma \ref{CU} we obtain
\[EX_1 =\mu=\gamma_1 =\frac{q}{2}\]
and, as we know \(\mathcal{B}_2=\frac{1}{6}\),
\[Var X_1 = \sigma^2=\gamma_2 =\frac{\mathcal{B}_2}{2}\left((q+1)^2-1\right)=\frac{q(q+2)}{12}. \]
\end{remark}

We now are ready to state the main theorem in form of a complete asymptotic expansion with explicit coefficients for the extended binomial coefficients \(\binom{n}{k}^{(q)}\).

\begin{theorem}\label{Main}For all integers \(N\geq 2\) we have
\[\sqrt{\frac{q(q+2) n}{12}}\frac{1}{(1+q)^n} \binom{n}{k}^{(q)}=\frac{1}{\sqrt{2 \pi}} e^{-x^2/2}+\sum_{\nu=1}^{\left[(N-2)/2\right]} \frac{q_{2\nu}(x)}{n^{\nu}} +o\left(\frac{1}{n^{(N-2)/2}}\right),\]
as \(n\rightarrow \infty\), uniformly with respect to all \(k\in\mathbb{Z}\), with
\[x=\frac{\sqrt{12}}{\sqrt{q(q+2) n}}\left(k-\frac{q}{2}n\right),\]
and
\begin{align}\label{q}&q_{2\nu} (x) = \frac{1}{\sqrt{2 \pi}} \left(\frac{12}{q(q+2)}\right)^{\nu} e^{-x^2 \slash 2}\\ \nonumber
&\times \sum_{k_2, k_4, \ldots, k_{2\nu} \geq 0 \atop k_2 + 2 k_4 + \cdots + \nu k_{2\nu} = \nu} ~
H_{2(\nu + s)} (x) \left(\frac{6}{q(q+2)}\right)^s \prod\limits^{\nu}_{m=1} ~ \frac{1}{k_{2m}!} 
\left( \frac{\mathcal{B}_{2(m+1)} \left((q+1)^{2m+2}-1\right)}{(2m+2)! (m+1) } \right)^{k_{2m}},
\end{align}
where \(s=k_2+k_4+\cdots+k_{2\nu}\).
\end{theorem}
\begin{proof} The proof is based on an application of Theorem \ref{EE} to the probabilities defined in (\ref{P}). First we observe that in our situation the functions given in (\ref{qq}) vanish identically for odd indices, which turns out to be a consequence of (\ref{GAMMA}). Indeed, if \(\nu=2l+1\) for an integer \(l \geq 0\), then in every solution \(k_1,\ldots, k_{2l+1} \geq 0\) of the equation
\[k_1 +2k_2+\cdots+(2l+1)k_{2l+1}=2l+1\]
there is at least one odd index \(i\) with \(k_i >0\). Consequently, using (\ref{GAMMA}) we have
\[\prod\limits^{2l+1}_{m=1} ~ \frac{1}{k_m!} 
\left( \frac{\gamma_{m+2}}{(m+2)! \sigma^{m+2}} \right)^{k_m}=0,\]
from which follows that \(q_{2l+1}(x)\) vanishes identically. Thus, only the functions \(q_{2\nu}(x)\) appear in (\ref{EE1}) and here we have
\[q_{2\nu} (x) = \frac{1}{\sqrt{2 \pi}} e^{-x^2 \slash 2} 
\sum_{k_1, \ldots, k_{2\nu} \geq 0 \atop k_1 + 2 k_2 + \cdots + 2\nu k_{2\nu} = 2\nu} ~
H_{2(\nu + s)} (x) ~ \prod\limits^{2\nu}_{m=1} ~ \frac{1}{k_m!} 
\left( \frac{\gamma_{m+2}}{(m+2)! \sigma^{m+2}} \right)^{k_m},
\]
where \(s = k_1 + \cdots + k_{2\nu}\). An analogous argument as in the odd case above shows that a solution \(k_1,\ldots,k_{2\nu}\) of the equation
\[k_1 +2k_2+\cdots+2\nu k_{2\nu}=2\nu\]
with a positive entry at an odd index does not give any contribution to the whole sum, so that we can write
\[q_{2\nu} (x) = \frac{1}{\sqrt{2 \pi}} e^{-x^2 \slash 2} 
\sum_{k_2, k_4, \ldots, k_{2\nu} \geq 0 \atop k_2 + 2 k_4+ \cdots + \nu k_{2\nu} = \nu} ~
H_{2(\nu + s)} (x) ~ \prod\limits^{\nu}_{m=1} ~ \frac{1}{k_{2m}!} 
\left( \frac{\gamma_{2m+2}}{(2m+2)! \sigma^{2m+2}} \right)^{k_{2m}},
\]
where \(s = k_2 + k_4+ \cdots + k_{2\nu}\). Now, taking the explicit form of the cumulants in (\ref{GAMMA}) into account, after some elementary computation we obtain (\ref{q}).

\end{proof}

\begin{remark} As a concluding remark we state the meaning of Theorem \ref{Main} for \(N=5\) explicitly. Using the known facts
\[H_4 (x)=x^4 -6x^2 +3,\quad \mathcal{B}_4=-\frac{1}{30},\]
we obtain 
\[\sqrt{\frac{q(q+2) n}{12}}\frac{1}{(1+q)^n} \binom{n}{k}^{(q)}=\frac{1}{\sqrt{2 \pi}} e^{-x^2/2}\left\{1-\frac{\left((q+1)^4-1\right)\left(x^4 -6x^2 +3\right)}{20nq^2 (q+2)^2}\right\} +o\left(\frac{1}{n^{3/2}}\right),\]
as \(n\rightarrow \infty\), uniformly with respect to all \(k\in\mathbb{Z}\), where we have
\[x=\frac{\sqrt{12}}{\sqrt{q(q+2) n}}\left(k-\frac{q}{2}n\right).\]
\end{remark}

\section{Acknowledgments} This work is supported by KU Leuven research grant OT\slash12\slash073 and the Belgian Interuniversity Attraction Pole P07/18.

\end{document}